\documentclass[12pt]{amsart}
\usepackage{amsfonts}
\usepackage{amsthm}
\usepackage{graphicx,subfigure}
\usepackage{color}
\usepackage{fullpage}
\theoremstyle{plain}
\newtheorem{thm}{Theorem}[section]
\newtheorem*{que2}{Question}
\newtheorem{lem}[thm]{Lemma}

\theoremstyle{definition}

\newtheorem{remark}[thm]{Remark}
\begin{document}
\title[Free Groups as  End Homogeneity Groups of $3$-manifolds]
{Free Groups as  End Homogeneity Groups of $3$-manifolds}
\author{Dennis J. Garity}
\address{Mathematics Department, Oregon State University,
Corvallis, OR 97331, U.S.A.}
\email{garity@math.oregonstate.edu}
\urladdr{http://www.math.oregonstate.edu/\symbol{126}garity}
\author{Du\v{s}an D. Repov\v{s}}
\address{Faculty of Education,
and Faculty of Mathematics and Physics,
University of Ljubljana\\
\& Institute of Mathematics, Physics and Mechanics,
Ljubljana SI-1000, Slovenia}
\email{dusan.repovs@guest.arnes.si}
\urladdr{http://www.fmf.uni-lj.si/\symbol{126}repovs}




\begin{abstract}
For every finitely generated free group $F$, we construct an
irreducible open $3$-manifold $M_F$ whose end  set is
homeomorphic to a Cantor set,  and with the end homogeneity group of
$M_F$ isomorphic to  $F$. The end homogeneity group is the
group of all self-homeomorphisms of the end set that extend to
homeomorphisms of the entire $3$-manifold. This extends an earlier result that constructs, for each finitely generated abelian group $G$, an irreducible open $3$-manifold $M_G$ with end homogeneity
group $G$. The method used in the proof of our main result also shows that if $G$ is a group with a  Cayley graph in $\mathbb{R}^3$ such that the graph automorphisms have certain nice extension properties,  then there is an irreducible open $3$-manifold $M_G$ with end homogeneity group $G$.
\end{abstract}
\maketitle

\section*{Introduction}
\label{introsection}

Each compact set $K$ in $S^{3}$ (or in $\mathbb{R}^3$) has for its complement an open
$3$-manifold $M^{3}$. There is often a close relation between properties of the embedding of $K$ in $S^3$ and properties of this $3$-manifold complement.   Consider, for example, embeddings of knots and the study of knot complements. See also  \cite{SoSt13} for a construction of a Cantor set with hyperbolic complement.

If the compact set $K$ is a  Cantor set $C$ in $S^{3}$, then the complement of $C$ is an open
$3$-manifold $M^{3}$  with end set $C$. The
embedding of the Cantor set gives rise to properties of the
corresponding complementary $3$-manifold $M^{3}$ and of the end set $C$  of $M^{3}$. See
\cite{GaRe13, GRW14, SoSt13} for examples of
this. There has been  extensive study of embeddings of Cantor sets in $S^{3}$ or $\mathbb{R}^3$. 
The space of embeddings is extremely complicated. 
 See for example \cite{GaKo13}.

In \cite{GaRe14}, the authors showed that for every finitely generated abelian group $G$, there is an irreducible open 3-manifold $M_G$ with end homogeneity group isomorphic to $G$. The end homogeneity group is the group of all end homeomorphisms that extend to the entire 3-manifold. This $3$-manifold is the complement of a carefully constructed Cantor set $C_G$ associated with $G$ in $S^3$. 

For further examples of properties of embeddings related to properties of the containing 3-manifold or of the 3-manifold arising as the complement, see \cite{Mo16} for a relation of certain group presentations to certain wild embeddings of arcs and  \cite{FlWu15} for a construction of certain wild Cantor sets that are Julia sets. See also \cite{GRW18} for another example of the use of embeddings to construct certain 3-manifolds.

In this paper, the authors show that for any finitely generated free group $F$, there is an irreducible open 3-manifold $M_F$ with the end homogeneity group $F$. The proof leads to a technique which shows that if  a group $G$ has a Cayley graph embeddable in $S^3$ in a particularly nice way, then there is an irreducible open 3-manifold $M_G$ with the end homogeneity group isomorphic to $G$. 

In Section \ref{prelimsection} we discuss concepts and results needed in the paper. This includes the concepts of geometric index, local genus, embedding homogeneity groups, end homogeneity groups, Antoine Cantor sets,  and rigidly embedded Cantor sets. In Section \ref{Cayleysection}, we carefully construct a specific Cayley graph in $\mathbb{R}^3$ for the free group on $N$ generators, $F_{N}$.
In Section \ref{Cantorsection} we construct a specific Cantor set in $S^3$ related to the Cayley graph constructed in Section \ref{Cayleysection}. 
In Section \ref{Theoremsection} we prove the main theorems. We end with some open questions.

\section{Preliminaries}
\label{prelimsection}

\subsection{Background Information}
We refer to \cite{GRZ05, GRZ06, GRWZ11, GRW14} for a discussion of
Cantor sets in general and of rigid Cantor sets, and to
\cite{Ze05} for results about local genus of points in Cantor
sets and defining sequences for Cantor sets.  The bibliographies
in these papers contain additional references to results about
Cantor sets. Two Cantor sets $X$ and $Y$ in $\mathbb{R}^3$ (or $S^3$) are  said to be
\emph{equivalent} or \emph{equivalently embedded} if there is a self-homeomorphism of $\mathbb{R}^3$ (or $S^3$) 
taking $X$ to $Y$. If there is no such homeomorphism, the Cantor
sets are said to be \emph{inequivalent}, or \emph{inequivalently
embedded}. A Cantor set $C$ is said to be \emph{rigidly embedded} in $\mathbb{R}^3$ (or $S^3$)
if the only self-homeomorphism of $C$ that extends to a
homeomorphism of $\mathbb{R}^3$ (or $S^3$) is the identity. 

\subsection{Geometric Index}

We list the results which  we will need on geometric index. See Schubert
\cite{Sc53, GRWZ11}  for more details.
If $K$ is a link in the interior of a solid torus $T$,  the \emph
{geometric index} of $K$ in $T$, denoted by $\text{N}(K,T)$, 
is the minimum  of $|K \cap D|$ over
all meridional disks  $D$ of $T$ intersecting $K$ transversely.   
 If $T$ is a solid torus and
$M$ is a finite union of disjoint solid tori so that $M \subset
\text{Int}  \ T$, then the geometric index $\text{N}( M,T)$ of
$M$ in  $T$ is $\text{N}(K,T)$ where $K$ is a core of $M$.

\textbf{Note:} An \emph{unknotted solid torus} is a solid torus $B^2 \times S^1$ with center line $\{0\}\times S^1$ an unknotted circle in $S^3$ .

\begin{thm}  {\rm(}\cite{Sc53}, \cite[Theorem
3.1]{GRWZ11}{\rm)} 
Let $T_0$ and $T_1$ be unknotted solid
tori in $S^{3}$ with  $T_0 \subset \rm{Int}  (T_1)$ and $\rm{N}(
T_0, T_1) = 1$.  Then $\partial T_0$ and  $\partial T_1$ are
parallel; i.e., the 3-manifold $T_1 -  \rm{Int}  (T_0)$ is
homeomorphic to $\partial T_0 \times I$ where $I$ is the closed
unit interval $[0,1]$. 
\end{thm}

\begin{thm}\label{productindex}  {\rm(}\cite{Sc53}, \cite[Theorem
3.2]{GRWZ11}{\rm)} 
 Let $T_0$ be a finite union of disjoint
solid tori. Let $T_1$ and $T_2$ be solid tori so that $T_0
\subset \rm{Int} ( T_1)$ and $T_1 \subset \rm{Int}  (T_2)$.  Then
$\rm{N}(T_0, T_2) =  \rm{N}(T_0, T_1) \cdot  \rm{N}(T_1, T_2)$.
\end{thm}

There is one additional result we will need.

\begin{thm}\label{evenindex}{\rm(}\cite{Sc53}, \cite[Theorem
3.3]{GRWZ11}{\rm)} 
Let $T$ be a solid torus in $S^{3}$ and
let $T_{1},\ldots T_{n}$ be unknotted pairwise disjoint solid tori in 
$T$, each of
geometric index $0$ in $T$. Then the geometric index of
$\bigcup\limits_{i=1}^{n}T_{i}$ in $T$ is even. 
\end{thm}

\subsection{Defining Sequences and Local Genus}

We review the definition and some basic facts from \cite{Ze05}
about  the local genus of points in a Cantor set. Also see \cite{Ze05}
for a discussion of defining sequences.

A \emph{defining sequence} for a Cantor set $X$ in $\mathbb{R}^3$ is a properly nested sequence of handlebodies $(M_i)$ 
such that  $\bigcap M_{i}=X$. Let
${ \mathcal{D}}(X)$ be the set of all defining sequences for a Cantor set $X\subset S^3$.
Let $(M_i)\in{\mathcal{D}}(X)$ be one of these defining sequences. For  $A\subset X$, denote by $M_i^A$
the union of those components of $M_i$ which intersect $A$.
Define 
\[ 
g_A(X;(M_i)) = \sup\{g(M_i^A);\ i\geq0\}\ \ \mbox{ and }
\]
\[\ g_A(X) = \inf\{ g_A(X;(M_i));\ (M_i) \in
{\mathcal{D}}(X)\}
\] 
where $g(M_i^A)$ is the genus of $M_i^A$. 
The number $g_A(X)$ is called \emph{the
genus of the Cantor set $X$ with respect to the subset $A$}. For
$A=\{x\}$ we call the number $g_{\{x\}}(X)$  \emph{the local
genus of the Cantor set $X$ at the point $x$} and denote it by
$g_x(X)$.

\begin{remark}The genus measures the minimum genus of handlebodies $M_i$ needed in any defining sequence for the Cantor set. The standard Cantor set has genus $0$  at each point since it can be defined as an intersection of genus $0$ balls. An Antoine Cantor set (illustrated in Figure \ref{Antoine}) has genus 1 at each point since it is defined by genus 1 manifolds (solid tori), but cannot be defined by genus 0 manifolds (balls).
\end{remark}
\begin{remark}
\label{genus}Let $x$ be an arbitrary point of a Cantor set $X$ and $h\colon
S^3\to S^3$ a homeomorphism. Then the local genus $g_x(X)$ is the
same as the local genus $g_{h(x)}(h(X))$. Also note that if $x\in
C\subset C^{\prime}$, then the local genus of $x$ in $C$ is less
than or equal to the local genus of $x$ in $C^{\prime}$.  See
\cite[Theorem 2.4]{Ze05}.
\end{remark}

The following result from \cite{Ze05} is needed to show that
certain points in our examples have local genus $> 1$.

\begin{thm}\cite{Ze05} \label{Slicing} 
Let $X,Y\subset S^3$ be
Cantor sets and $p\in X\cap Y$.  Suppose there exists a 3-ball
$B$ and a 2-disc $D\subset B $ such that

\begin{enumerate}
\item $p\in\rm{Int} B$, $\partial D=D\cap\partial B$, $D\cap (X\cup
Y)=\{p\}$; and

\item $X\cap B\subset B_X\cup\{p\}$ and $Y\cap B\subset
B_Y\cup\{p\}$ where $B_X$ and $B_Y$ are the components of
$B -  D$.
\end{enumerate}

Then $g_p(X\cup Y)=g_p(X)+g_p(Y)$.
\end{thm}

\subsection{Embedding Homogeneity Groups and End Homogeneity Groups}
For background on Freudenthal compactifications and theory of
ends, see \cite{Di68, Fr42, Si65}. For an
alternate proof using defining sequences of the result that every
homeomorphism of the open $3$-manifold extends to a homeomorphism
of its Freudenthal compactification, see \cite{GaRe13}.

Each Cantor set $C$ in $S^{3}$ has for its complement an open
$3$-manifold $M^{3}$  with end set $C$. The
embedding of the Cantor set gives rise to properties of the
corresponding complementary $3$-manifold $M^{3}$. See
\cite{GaRe13, GRW14, SoSt13} for examples of
this.

We investigate possible group actions on the end set $C$ of the
open $3$-manifold $M^{3}$ in the following sense:
The
\emph{homogeneity group of the end set }is the group of
homeomorphisms of the end set $C$ that extend to homeomorphisms
of the open $3$-manifold $M^{3}$. Referring specifically to the
embedding of the Cantor set, this group can also be called the
\emph{embedding homogeneity group of the Cantor set}. See 
\cite{Di11, vM11} for a discussion and overview of some
other types of homogeneity.

The standardly embedded Cantor set is at one extreme here. The
embedding homogeneity group is the full group of 
self-homeomorphisms of the Cantor set, an extremely rich group (there
is such a homeomorphism taking any countable dense subset to any
other). Cantor sets with this full embedding homogeneity group
are called \emph{strongly homogeneously embedded.} See Daverman
\cite{Da79} for an example of a non-standard Cantor set with this
property.

At the other extreme are \emph{rigidly embedded} Cantor sets,
i.e. those Cantor sets for which only the identity homeomorphism
extends. Shilepsky \cite{Sh74} constructed Antoine type
\cite{An20} rigid Cantor sets. Their rigidity is a consequence of
Sher's result \cite{Sh68} that if two Antoine Cantor sets are
equivalently embedded, then the stages of their defining
sequences must match up exactly. Newer examples
\cite{GRZ06, GRWZ11} of non-standard Cantor sets were
constructed  that are both rigidly embedded and have simply
connected complement. See \cite{Wr86} for  additional examples of rigidity.

These examples naturally lead to the question of which types of
groups can arise as end homogeneity groups between the two
extremes mentioned above. In this paper we show that for each
finitely generated free group $G$, there is an irreducible
open 3-manifold with the end set homeomorphic to a Cantor set
and with the end homogeneity group isomorphic to $G$.  

\begin{remark}
\label{irreducible}The Cantor sets produced are \emph{unsplittable} as are the set
of ends, in the sense that no $2$-sphere separates points of the
Cantor set (respectively, points of the end set). Correspondingly, the complements of the Cantor sets produced are irreducible in the sense that every 2-sphere in the complement bounds a 3-ball.
\end{remark}

\subsection{Antoine Cantor Sets}
An Antoine Cantor set is described as follows. Let $M_{0}$ be an
unknotted solid torus in $S^{3}$. Let $M_{1}$ be a chain of at
least four linked, pairwise disjoint, unknotted solid tori in $M_{0}$ 
as in Figure
\ref{Antoine}. Inductively, $M_{i}$ consists of pairwise
disjoint solid tori in $S^{3}$ and $M_{i+1}$ is obtained from
$M_{i}$ by placing a chain of at least four linked, pairwise disjoint, unknotted
solid tori in each component of $M_{i}$. If the diameter of the
components goes to $0$, the Antoine Cantor set is
$C=\bigcap\limits_{i=0}^{\infty}M_{i}$.

\begin{figure}[h]
\begin{center}
    \includegraphics[width=0.45\textwidth]{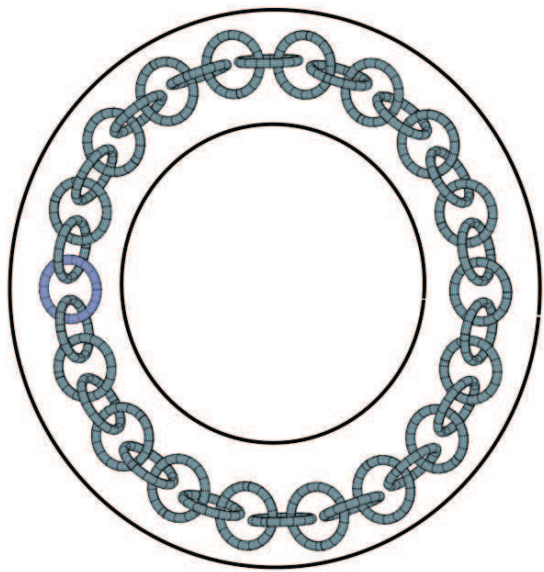}
\end{center}
 \caption{%
    Two Stages for an Antoine Construction}%
\label{Antoine}
\end{figure}

We refer to Sher's paper \cite{Sh68}  for basic results and
description of Antoine Cantor sets. The key result we will need is the
following:
\begin{thm}\cite[Theorems 1 and 2]{Sh68}\label{Sher} 
Suppose $C$
and $D$ are Antoine Cantor sets in $S^{3}$ with defining
sequences $\{M_{i}\}$ and $\{N_{i}\}$, respectively. Then $C$ and $D$   are equivalently embedded if and only if there is a 
self-homeomorphism $h$ of $S^{3}$ with $h(M_{i})=N_{i}$ for each $i$.
\end{thm}

In particular, the number and adjacency of links in the chains
must match up at each stage.
\begin{remark}
\label{separation}
A standard argument shows that Antoine Cantor set cannot be
separated by any $2$-sphere. 
\end{remark}

\begin{remark}
\label{isotopy} Also note that the homeomorphism in Theorem
\ref{Sher} can be realized as the final stage of an ambient  isotopy in $S^3$ since
each of the homeomorphisms in the argument can be realized by an
ambient isotopy.
\end{remark}

\section{A Cayley Graph for $F_{N}$}
\label{Cayleysection}
Let $F_{N}$ be the free group on generators, $a_1,\ldots,a_N$. Denote the inverse of $a_i$ by $\overline{\mathstrut a_i}$ and let $g(F_{N})$, the full generating set of $F_{N}$, be $\{a_1,\ldots,a_N, \overline{\mathstrut a_1},\ldots, \overline{\mathstrut a_N}\}$.  For economy of notation, let $\beta_j = \alpha_j$ for $1\leq j \leq N$ and let $\beta_j =\overline{\mathstrut \alpha_{j-N}}$ for $N+1\leq j \leq 2N$. Note that $\overline{\mathstrut \beta_k} = \beta_{k+n \text{ (mod 2N)} }$. Elements of $F_{N}$ correspond to reduced words using elements of $g(F_{N})$ as letters. Group multiplication corresponds to concatenation and reduction of words. We first describe a specific embedding for the Cayley graph $G(F_{N})$ associated with this presentation of $F_{N}$.

The identity is represented by the empty word and corresponds to a vertex $v(e)$ at the origin.
There are $\sigma(n)=2N\cdot (N-1)^{n-1}$ reduced words of length $n$. These will be represented by $\sigma(n)$ equally spaced vertices on the circle of radius $n$ centered at the origin, 
at radial angles:\\
\centerline{
$\{
n(k)\equiv \dfrac{2\pi}{\sigma(n)} k -\dfrac{2 \pi}{2\sigma(n)} 
\ \vert\ 
1\leq k\leq \sigma(n)
\}$. 
}
Each of these points, $p(n,k)\equiv n\cdot e^{ n_k\cdot i}$ for $ 1\leq k\leq \sigma(n)$, on the circle of radius $n$  will be joined by edges to $N-1$ vertices on the circle of radius $n+1$. Specifically, 
$p(n,k)$ will be joined to the points:\\
\centerline{
$\{
p(n+1,j)
\ \vert\ 
(k-1)(N-1)< j \leq (k)(N-1)
\}.$ 
}

Figure \ref{GraphFig} show the  first two stages of the Cayley graph for $F_{2}$ and $F_{3}$. The colored directed edges correspond to generators and will be described in more detail in the next subsection.

\begin{figure}[h]
\begin{center}
  \subfigure[$F_{2}$ Stage 2]%
    {%
    \label{GraphFig22}
    \includegraphics[width=0.4\textwidth]{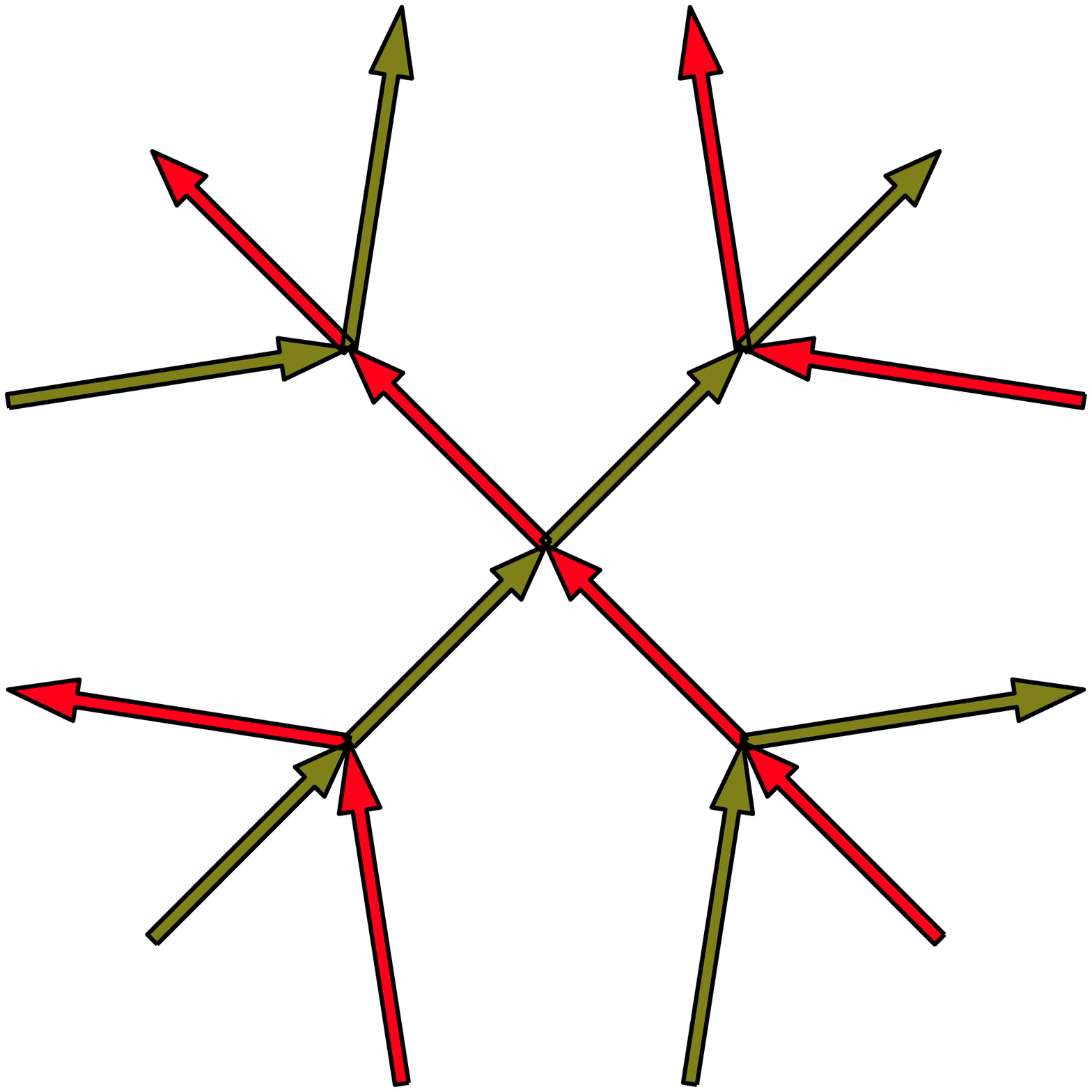}
    }%
  \subfigure[$F_{2}$ Stage 3]%
    {%
    \label{GraphFig23}
    \includegraphics[width=0.4\textwidth]{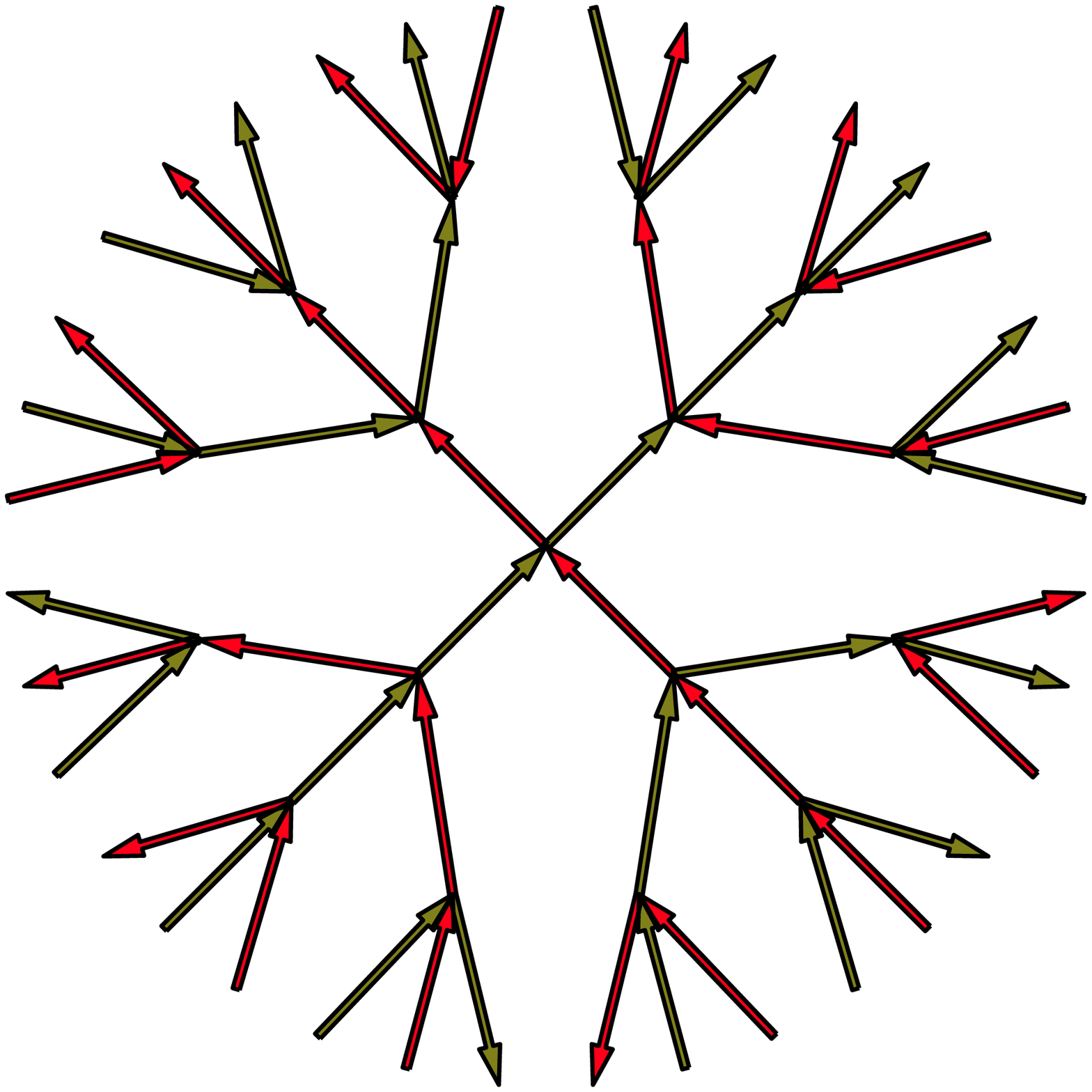}
     }\\  
  \subfigure[$F_{3}$ Stage 2]%
    {%
    \label{GraphFig32}
    \includegraphics[width=0.4\textwidth]{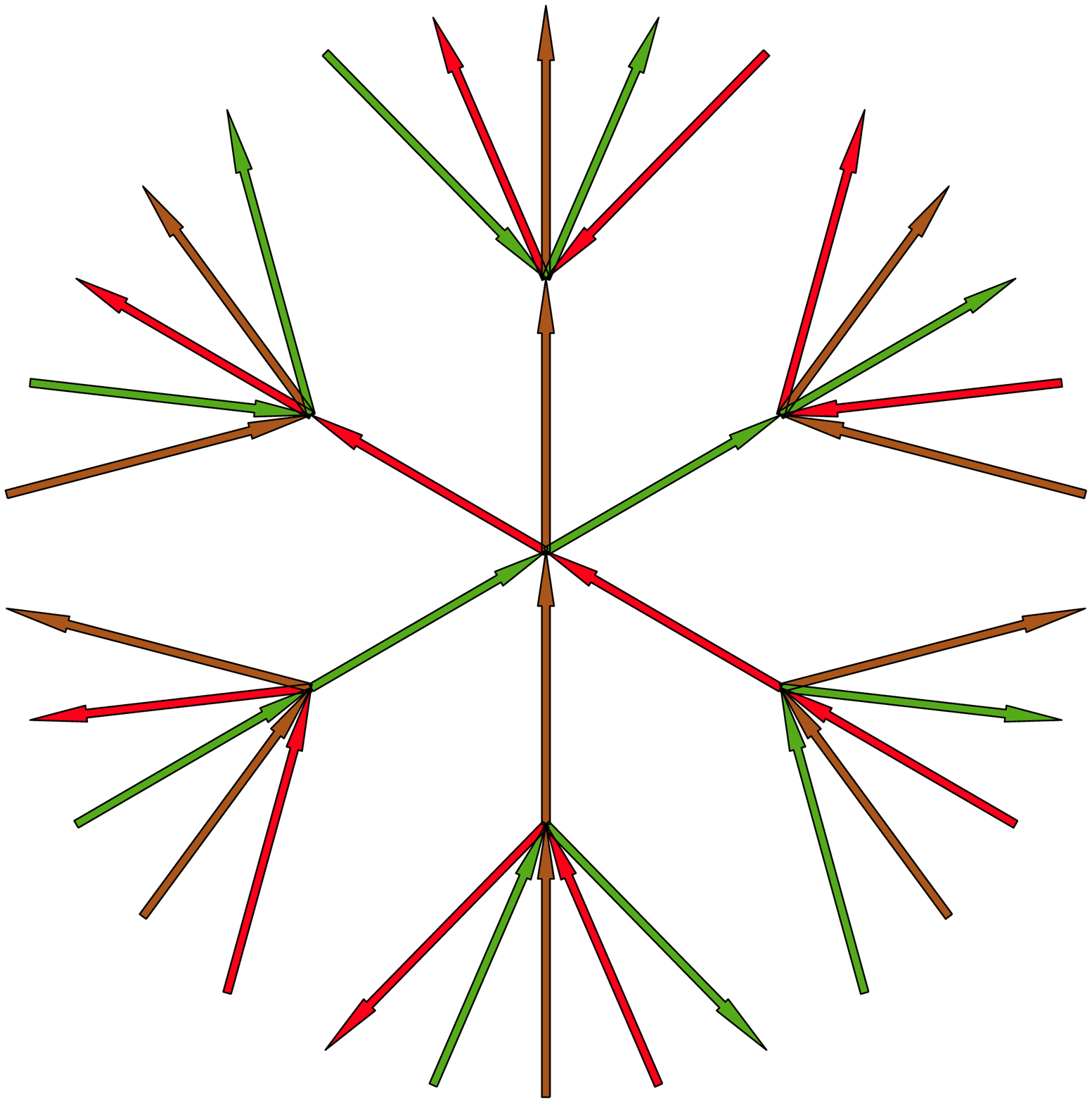}
    } %
  \subfigure[$F_{3}$ Stage 3]%
    {%
    \label{GraphFig33}
    \includegraphics[width=0.4\textwidth]{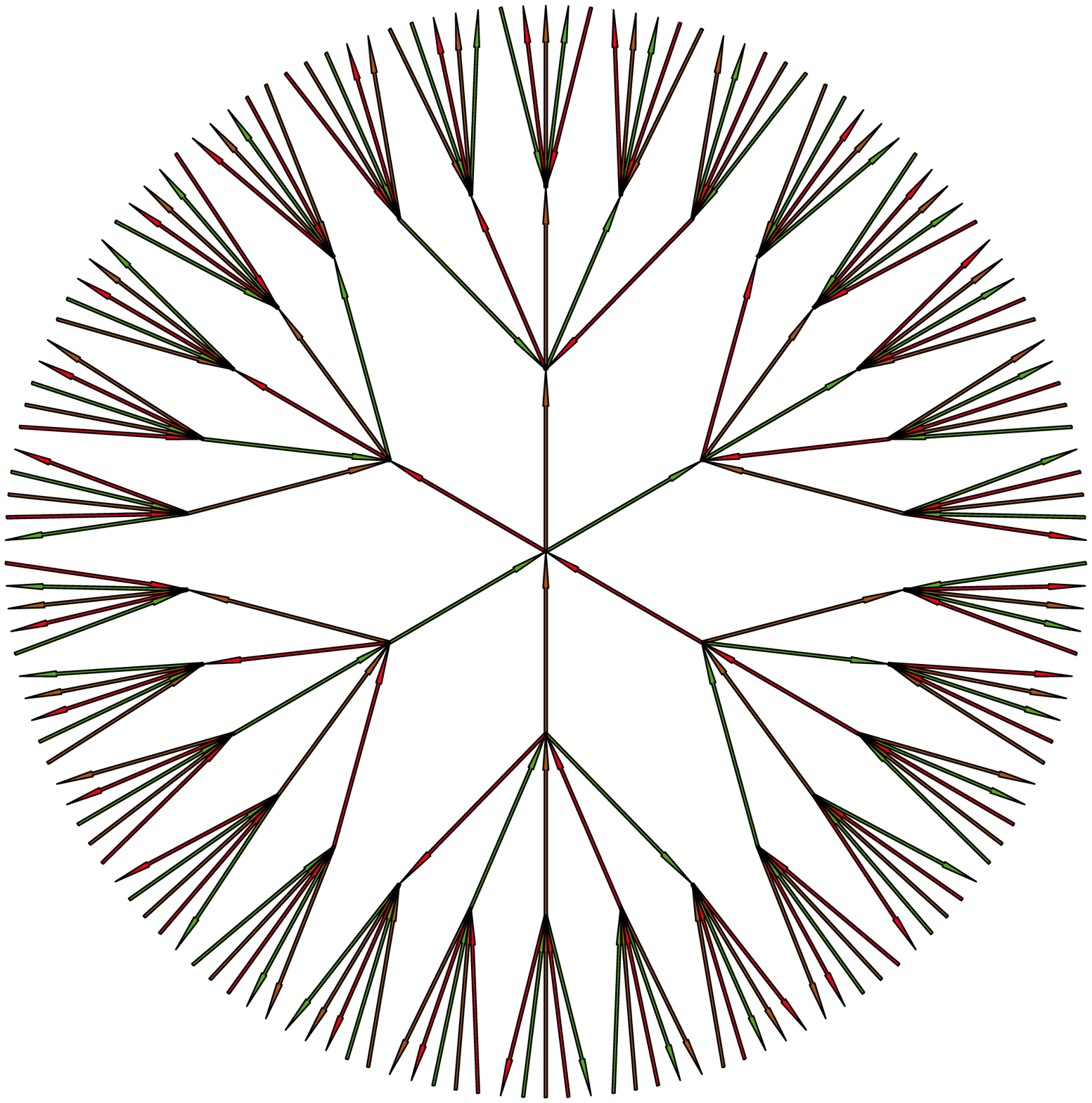}
    } 
 \end{center}
 \caption{%
    Cayley Graphs for $F_{2}$ and $F_{3}$}%
\label{GraphFig}
\end{figure}

\subsection{Labeling Edges and Vertices}

We inductively describe how to label the vertices and how to label and orient the edges of the graph described above. 
\begin{remark}
Note that an edge oriented from $v$ to $w$ labeled with the generator $a_i$ is regarded the same as that edge oriented from $w$ to $v$ and labeled with the generator $\overline{\mathstrut a_i}$.
\label{orientededges}
\end{remark}
Begin by labeling the vertex at the origin corresponding to the identity as $v(e)$. Label the vertex at $p(1,j)$ for $1\leq j \leq 2N$ as $v(\beta_j)$.

Label the edge from $v(e)$ to $v(\beta_j)$, $1\leq j \leq 2N$ with the generator $\beta_j$ and orient it from $v(e)$ to $v(\beta_j)$.
Note that this is equivalent to labeling the edge from $v(\beta_j)$ to $v(e)$  with the generator $\overline{\mathstrut \beta_j} $ and orienting it from  $v(\beta_j)$ to $v(e)$. For notation, we refer to this edge as $E(v(e), \beta_j)$ or $E(v(\beta_j),\overline{\mathstrut \beta_j})$ to emphasize the initial vertex and generator needed to get to the other vertex. To emphasize the initial and final vertices, we can use the notation $E(v(e), v(\beta_j))=E(v(\beta_j), v(e))$.

Note that starting at any point on the circle of radius $1 \slash 2$ centered at the origin and traveling in the counterclockwise direction around this circle, one encounters the labeled edges in the cyclic order  $E(v(e), \beta_1),\ldots , E(v(e), \beta_{2N})$, or equivalently, by Remark \ref{orientededges}, in the cyclic order:\\
\centerline{$E(v(e), \beta_1),\ldots , E(v(e), \beta_{N}), E(v(\beta_{N+1}), \beta_{1}),\ldots , E(v(\beta_{2N}), \beta_N)$.}

Inductively assume the vertices at distance $m$ from the origin, $m\leq n$, have been labeled as $v(g_i)$, where  $g_i$ is a reduced word of length $m$. Also assume that edges from vertices $v(g_i)$ at distance $m$ from the origin $(m\leq n-1)$ to $v(g_\ell)$ at distance $m+1$ from the origin have been labeled $E(v(g_i), \beta_k)$ where
$g_\ell = g_i\circ \beta_k$. 

Consider $v(g)$ for some $g$ of reduced length n. Then $g=g^{\,\prime}\circ \beta_k$ for some reduced word $g^{\,\prime}$ of length $n-1$ and some $\beta_k \in g(F_{N})$. The edge joining $v(g^{\,\prime})$ to $v(g)$ is labeled 
$E(v(g^{\,\prime}), \beta_k)$, or equivalently $E(v(g), \overline{\mathstrut \beta_k}) = 
E(v(g), \beta_{k+N \text{(mod 2N)}})$. Let $j=k+n \text{(mod 2N)}$.

Form a circle of radius $1 \slash 2$ about $v(g)$. Starting at the edge $E(v(g), \beta_j)$, proceed around the circle in the counterclockwise direction labeling the edges encountered as 
$$E(v(g), \beta_{j+1}), \ldots 
E(v(g), \beta_{j+2N-1}) ,$$ all subscripts mod(2N).

\begin{figure}[h]
\begin{center}
  \subfigure%
    {%
    \label{Labelling1}
    \includegraphics[width=0.47\textwidth]{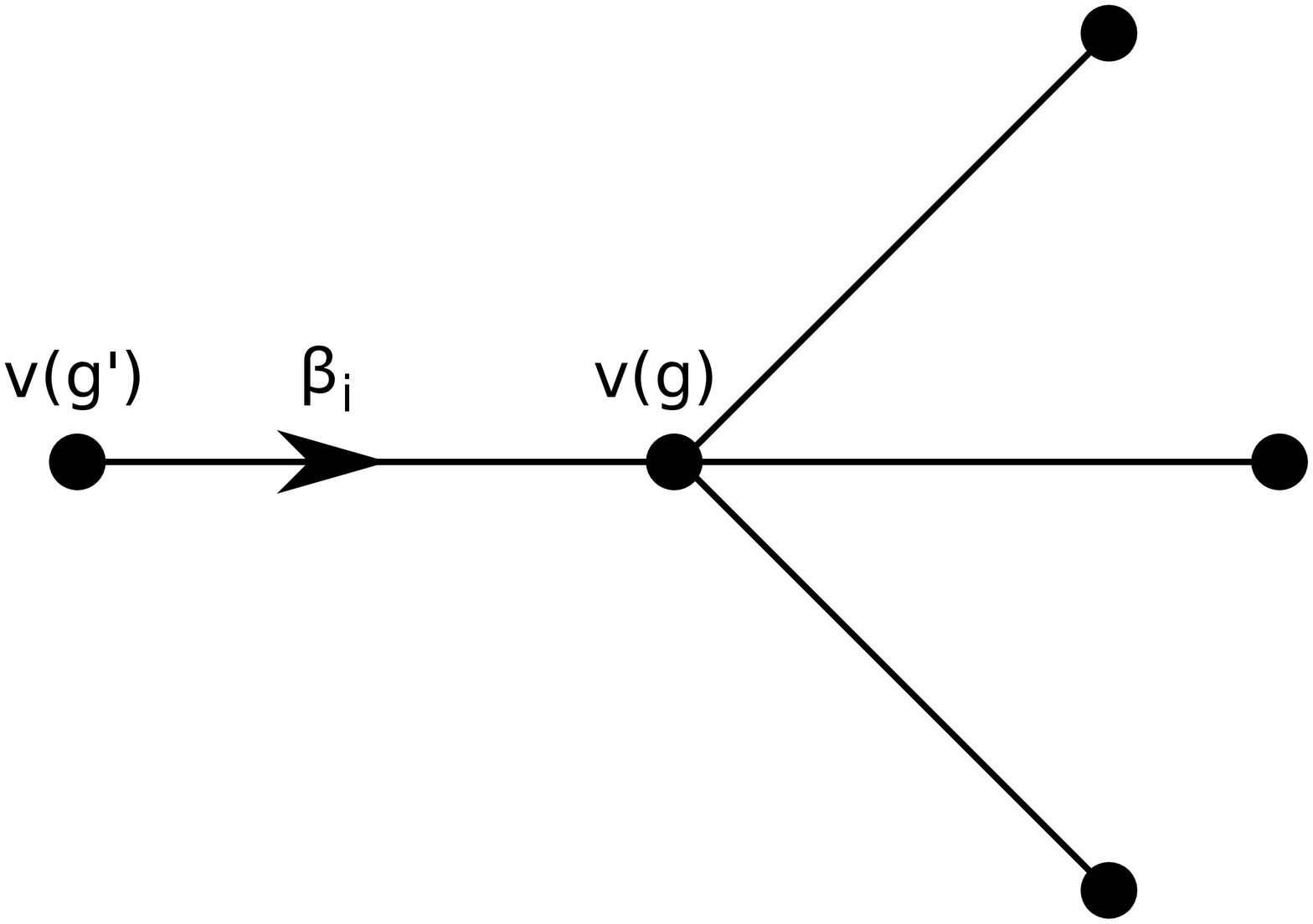}
    }%
  \subfigure%
    {%
    \label{Labelling2}
    \includegraphics[width=0.47\textwidth]{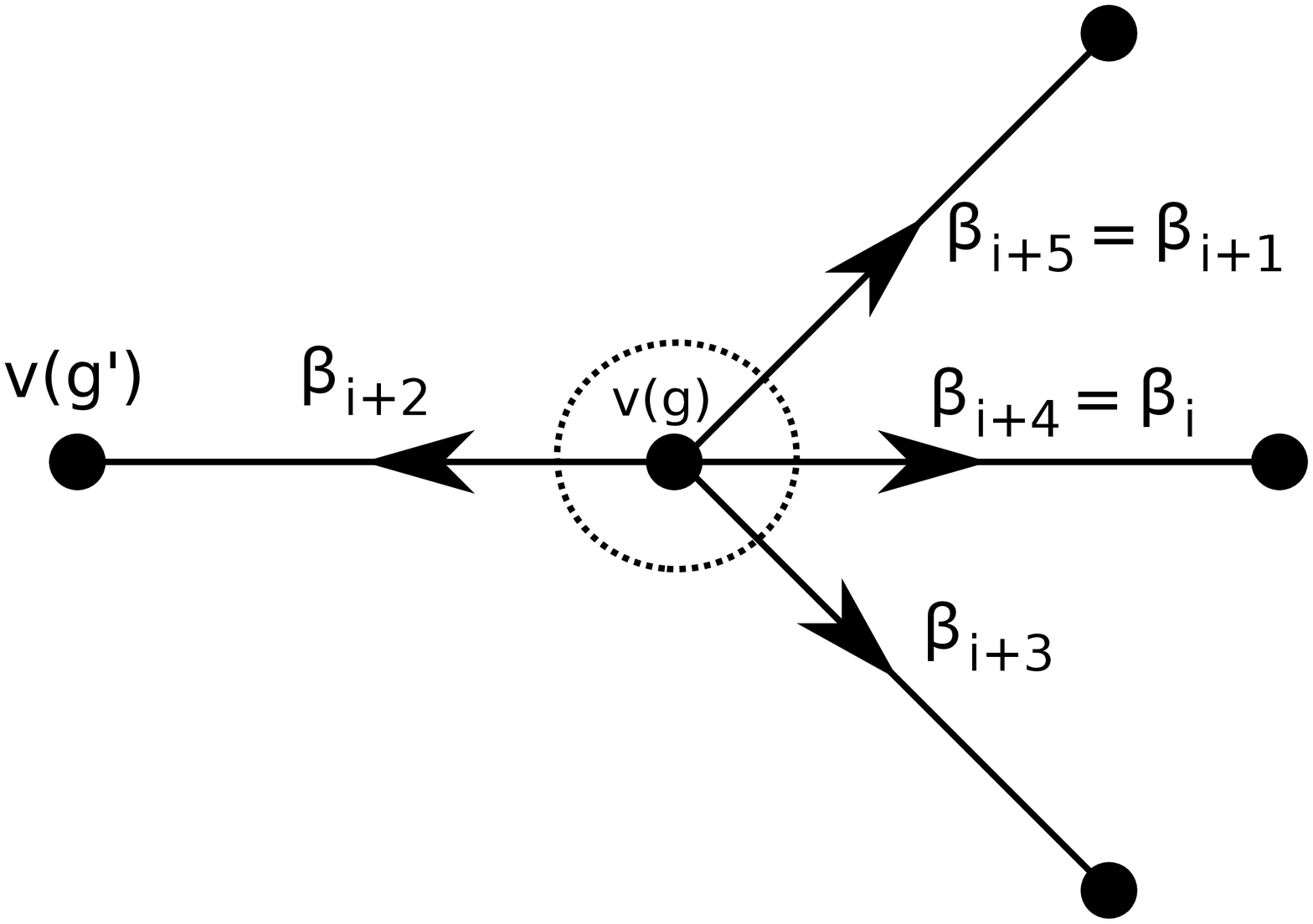}
     }  
  \end{center}
 \caption{%
    Inductively Labelling Edges}%
\label{Labelling}
\end{figure}

Label the vertices at distance $n+1$ from the origin as follows. If $v$ is such a vertex, then there is an edge
$E(v(g), \beta_{j})$ from a vertex $v(g)$ at distance $n$ from the origin to $v$. Label the vertex $v$ as $v(g\circ \beta_j)$.

The inductive assumption is easily checked for the newly labeled vertices and edges.
See Figure \ref{Labelling} for an illustration, keeping in mind Remark \ref{orientededges}.

\subsection{Cayley Graph Automorphisms for ${\pmb{G(F_{N})}}$} 

For each element $g$ of $F_{N}$ we define a graph automorphism $h_{g}$ of $G(F_{N})$ as follows. For a vertex 
$v({g_i})$ of $G(F_{N})$ corresponding to the element $g_i$ of $F_{N}$, $h_{g}(v({g_i}))$ is defined to be $v(g\circ g_i)$. The edge from $v(g_i)$ to $v(g_i\circ \beta_k)$  can be labeled $E(v(g_i), \beta_k)$ when oriented from $v(g_i)$ to $v(g_i\circ \beta_k)$. Define $h_{g}(E(v(g_i), \beta_k)$ to be the linear homeomorphism from  $ E(v(g_i), \beta_k)$ to $E(v(g\circ g_i), \beta_k)$. 

\begin{remark}\label{associativity}Note that these graph automorphisms define homeomorphisms from $G(F_{N})$ to itself. Note also that $h_{g_1}\circ h_{g_2}=h_{g_1\circ g_2}$.
\end{remark}

\subsection{Extending Graph Automorphisms to Homeomorphisms of  $\mathbb{R}^3$} 
 The  graph automorphisms $h_{\beta_i}$, for $1\leq i\leq N$, can be extended to homeomorphisms ${h^\prime}_{\beta_i}$ from $\mathbb{R}^2$ to itself. Note that $\mathbb{R}^2-G(F_{N})$ consists of a collections of unbounded regions with boundary in $G(F_{N})$. Adding in the point at infinity, each region is a closed disc with boundary consisting of edges in $G(F_{N})$ together with this added point. The graph automorphism $h_{\beta_i}$ takes the boundary of one of these disc regions to the boundary of another, and can thus be extended to the entire region.  
 There are many choices for these extension homeomorphisms. They coincide up to proper isotopy. 
 
 Piecing together these extensions, one gets an extension of $h_{\beta_i}$ to a self homeomorphism of $\mathbb{R}^2$,  ${h^\prime}_{\beta_i}$ for $1\leq i\leq N$. Define ${h^\prime}_{\beta_i}$ for $N+1\leq i\leq 2N$ to be $ (h^{\prime}_{\beta_{i-N}})^{-1}$. Since $h_{\beta_i}= ( h_{\beta_{i-N}})^{-1}$ for $N+1\leq i \leq 2N$ as graph automorphisms, it follows that $h^\prime _{\beta_i}$ is an extension of $h_{\beta_i}$ for $N+1\leq i\leq 2N$.

For a general $g\in F_{N}$, $g$ has a unique reduced representation $g=\beta_{i1}\circ \beta_{i2}\circ \ldots \circ \beta_{iM}$ where the $\beta_{ij}$ are in $g(F_{N})$. Define $h^\prime (g)$ to be: 
$$
h^\prime_g=h^\prime_{\beta_{i1}}\circ h^\prime_{\beta_{i2}}\circ \ldots \circ h^\prime_{\beta_{iM}}
$$

The homeomorphism $h^\prime_g:\mathbb{R}^2\rightarrow \mathbb{R}^2$ can then be extended to a homeomorphism 
 $H_g:\mathbb{R}^3\rightarrow \mathbb{R}^3$ by crossing with the identity.

\begin{remark}\label{associativity R3}Note  that $h^\prime_{g_1}\circ h^\prime_{g_2}=h^\prime_{g_1\circ g_2}$. This follows from the uniqueness of reduced representations of elements of $F_{N}$. The definition of $H_g$ as the product of $h^\prime_g$ with the identity shows that it is also true that  $H_{g_1}\circ H_{g_2}=H_{g_1\circ g_2}$.
\end{remark}

\section{A Cantor Set $C(F_{N})$ with Embedding Homogeneity Group $F_{N}$}
\label{Cantorsection}

\subsection{A Collection of Balls and Tubes in $\mathbb{R}^3$ Containing $G(F_{N})$} 

In $\mathbb{R}^3$, choose a 3-ball $B(v(e))$ around the origin $=v(e)$ so that the homeomorphic copies 
$H_{\beta_i}(B(v(e)))$, $1\leq i \leq 2N,$ together with $B(v(e))$ are all pairwise disjoint. Let $B(v(\beta_i))$ denote the 3-ball $H_{\beta_i}(B(v(e)))$.
The original ball $B(v(e))$ can be chosen so that it is contained in a ball of radius $1\slash 5$ about $v(e)$, and so that each $B(v(\beta_i))$ is contained in a ball of radius $1\slash 5$ about $v(\beta_i)$.

For a general $g\in F_{N}$, use the unique reduced representation of  $g$ as $\beta_{i1}\circ \beta_{i2}\circ \ldots \circ \beta_{iM}$ where the $\beta_{ij}$ are in $g(F_{N})$. The ball $B(v(g))$ is then defined to be to be:
$$
H_{\beta_{i1}}\circ H_{\beta_{i2}}\circ \ldots \circ H_{\beta_{iM}}(B(v(e)))=H_{g}(B(v(e))).
$$ 
We similarly define a collection of tubes. For each $i$, $1\leq i \leq N$ choose a thin tube $D(v(e),\beta_i)$ joining $B(v(e))$ to $B(v(\beta_i))$ so that 
$B(v(e))\cup D(v(e),\beta_i) \cup B(v(\beta_i))$ 
contains $E(v(e), \beta_i)$. Choose the tubes $D(v(e),\beta_1), \ldots , D(v(e),\beta_N)$, so as to be pairwise disjoint. For each $i$, $N+1\leq i \leq 2N$ let $D(v(e),\beta_i)$ be 
$H_{\beta_i}(D(v(e),\beta_{i-N}))=H_{\beta_i}(D(v(e),\overline{\beta_{i}}))$.
See Figure \ref{BallsTubes} for an illustration for $F_{2}$.

\begin{remark} Note that the relationship $D(v(e),\beta_i) = H_{\beta_i}(D(v(e),\overline{\beta_{i}}))$ then holds for all $i$, $1\leq i \leq 2N$.
\end{remark}

\begin{figure}[h]
\begin{center}
    \includegraphics[width=0.45\textwidth]{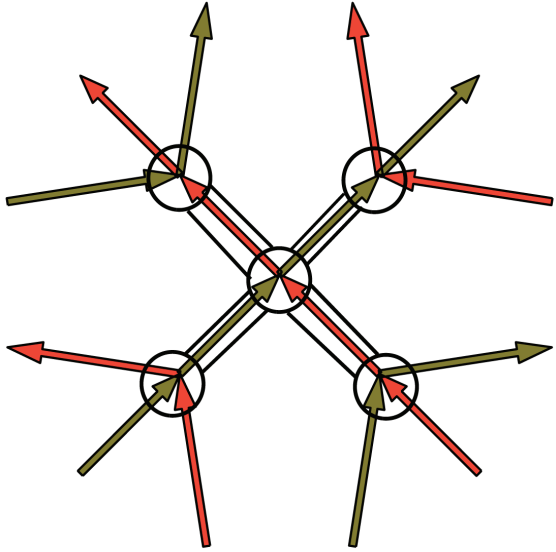}
\end{center}
 \caption{%
    Balls, Tubes around $G(F_{2})$ }%
\label{BallsTubes}
\end{figure}


\subsection{Rigid Cantor Sets Using $G(F_{N})$ as a Guide.} 


\begin{figure}[h]
\begin{center}
    \includegraphics[width=0.47\textwidth]{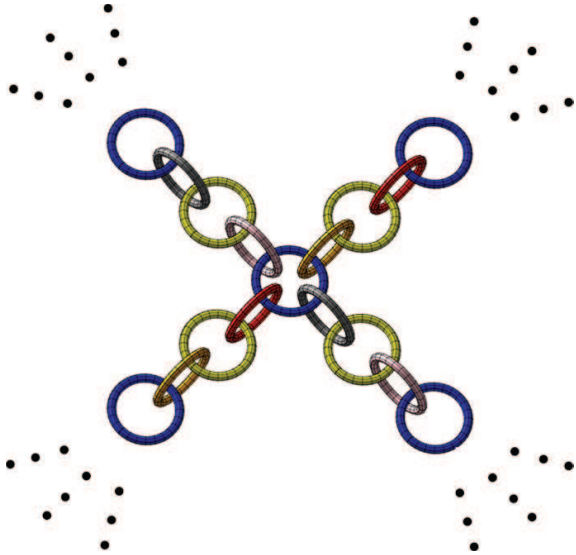}
\end{center}
 \caption{%
    Construction of Rigid Cantor Sets Placement for $G(F_{2}$)}%
\label{Rigid}
\end{figure}

We will now define  rigid Antoine Cantor sets  associated with each vertex in $G(F_N)$, and a chain of three linked rigid Cantor sets for each edge in $G(F_n)$. See Figure \ref{Rigid} for an illustration of the outer tori in defining sequences for these rigid Cantor sets in the case of $F_2$.

To begin, construct a rigid Antoine Cantor set $C(v(e))$  so the outer torus, $T(C(v(e)))$,  of the construction is  in $B(v(e))$. This $T(C(v(e)))$ corresponds to the central torus in Figure \ref{Rigid} in the case of $F_{2}$. Let $T^{\prime}(C(v(e)))$ be the union of the second stage tori in the construction of $C(v(e))$. See Figure \ref{Antoine} for an illustration of the first and second stages of an Antoine construction.

For a general $g\in F_{N}$, use the unique reduced representation of  $g$ as $\beta_{i1}\circ \beta_{i2}\circ \ldots \circ \beta_{iM}$ where the $\beta_{ij}$ are in $g(F_{N})$. Define a rigid Cantor set $C(v(g))\subset B(v(g))$  to be:
$$
H_{\beta_{i1}}\circ H_{\beta_{i2}}\circ \ldots \circ H_{\beta_{iM}}(C(v(e)))=H_{g}(C(v(e))).
$$ 

\begin{remark}
Note that $C(v(e))$ and any $C(v(g))$ are equivalently embedded.
\end{remark}

Similarly, let 
$T(C(v(g)))=H_{g}(T(C(v(e)))) \text{ and } T^{\,\prime}(C(v(g)))=H_{g}(T^{\prime}(C(v(e)))).
$
The four outer tori in Figure \ref{Rigid} correspond to $T(C(a_1)), T(C(a_2)), T(C(\overline{\mathstrut a_1}))$,
and 
$TC(\overline{\mathstrut a_2})$ in the case of $F_{2}$.

We next construct rigid Cantor sets  
$C_{1}(v(e),\beta_{k}), C_{2}(v(e),\beta_{k}), C_{3}(v(e),\beta_{k})$ for each $k, 1\leq k\leq N$. 
Choose a series of three linked tori, 
$T_{1}(v(e),\beta_{k}), T_{2}(v(e),\beta_{k}),$ and $T_{3}(v(e),\beta_{k})$ 
joining  
$T(C(v(e))) \text{ and } T(C(\beta_k)), 1\leq k\leq N$ 
and contained in 
$B(v(e))\cup D(v(e),\beta_k) \cup B(v(\beta_k))$ 
as in Figure \ref{Rigid}. 
Label the torus linked to $T(C(v(e)))$ as $T_{1}(v(e),\beta_{k})$,  and use it as the outer torus for a rigid Antoine Cantor set $C_{1}(v(e),\beta_{k})$. 
Label the torus linked to  $T(C(\beta_k))$ as $T_{3}(v(e),\beta_{k})$, and use it as the outer torus for a rigid Antoine Cantor set $C_{3}(v(e), \beta_{k})$. 
Label  the remaining torus as $T_{2}(v(e),\beta_{k})$ and use it as the outer torus for a rigid Antoine Cantor set $C_{2}(v(e), \beta_{k})$. 
\begin{remark}
\label{inequivalent}
Choose these $3N$ rigid Cantor sets $\{T_{i}(v(e),\beta_{k}), 1\leq 1\leq 3, 1\leq k \leq N\}$ so that they are all inequivalently embedded, and so that they are all inequivalent to $C(v(e))$.
\end{remark}

For $N+1\leq k \leq 2N$, and $1\leq j \leq 3$, 
let $C_{j}(v(e), \beta_{k})=H_{\beta_k}(C_{j}(v(e), \overline{\beta_{k}}))$, 
and 
$T(C_{j}(v(e), \beta_{k}))=H_{\beta_k}(T(C_{j}(v(e), \overline{\beta_{k}})))$.
This completes the definition of the $3\cdot2N$ rigid Cantor sets  along the edges $E(v(e),\beta_{k})$. See Figure \ref{Rigid} for outer tori corresponding to these Cantor sets in the case $F(2)$.

For $g\in F_{N}$, let $C_{i}(v(g),\beta_k)=H_{g}(C_{i}(v(e), \beta_k))$ 
and  $T(C_{i}(v(g),\beta_k))=H_{g}(T(C_{i}(v(e), \beta_k)))$
for $1\leq i \leq 3$ and  $1\leq k \leq 2N$.

For any  rigid Antoine Cantor set $C$ in this construction, $T(C)$ represents the outer torus in the Antoine Construction and $T^{\prime}(C)$ represents the union of the second stage tori in the construction.

\begin{remark}
Because $E(v(g),\beta_k) \text{ and } E(v(g\circ \beta_k), \overline{\beta_k})$ are identified, it appears that there are two possible definitions for the $C_{i}$ on this edge: $H_{g}(C_{i}(v(e), \beta_k))$ and 
$H_{g\circ \beta_k}(C_{i}(v(e),\overline{\beta_k}))$. 
However, 
$
H_{g\circ \beta_k}(C_{i}(v(e),\overline{\beta_k}))
=H_{g}(
H_{\beta_{k}}(
 H_{\,\overline{\beta_k}}(C_{i}(v(e), \beta_{k_{i}}))))=H_{g}(C_{i}(v(e), \beta_k)).
$
\end{remark}

\begin{remark}
\label{inequivalent2}
Note
that none of the $C_{i}(v(g),\beta_{k})$ are equivalent to any $C(v(g))$, and that  
$C_{i}(v(g),\beta_{k})$ is equivalent to $C_{j}(v(g^{\,\prime}),\beta_{\ell})$ 
if and only if one of the following two conditions holds:

\centerline{$ i=j \text{ and } \beta_{k}=\beta_{\ell}\text{, or }
i=j \text{ and } \beta_{k}=\overline{\beta_{\ell}}.$
}
\end{remark}
The union of the Cantor sets $C_{i}(g,\beta_k)$ together with the union of the Cantor sets $C(v(g))$ is not a Cantor set since it is not compact. To remedy this, we add the point at infinity, denoted by $p$, to $\mathbb{R}^3$ to get $S^3$. We then define the Cantor set $\mathcal{C}$ associated with $G(F_{N})$ to be:
$
\mathcal{C}=C(G(F_{N}))= \{C_{i}(v(g),\beta_k), g\in F_{N}, 1\leq k \leq 2N\} \cup \{C(v(g)), g\in F_{N}\}\cup\{p\}.
$

\section{Main Theorems}
\label{Theoremsection}
We now have $\mathcal{C}\subset S^3$. One can show that $\mathcal{C}$  cannot be separated by any 2-sphere, by an argument similar to that referred to in Remark \ref{separation}. Let $\mathcal{M}=M(F_{N})$ be the 3-manifold $S^3- \mathcal{C}$. By Remark \ref{irreducible}, $\mathcal{M}$ is irreducible.

We will  show that the embedding homogeneity group of $\mathcal{C}\subset S^3$ is $F_{N}$, and correspondingly, that the end homogeneity group of $\mathcal{M}$ is $F_{N}$.

We  use the following three technical lemmas in the proof of Theorem \ref{MT 2} below. In reading the first lemma, think of $T_1$ and $T_2$ as two linked tori in a stage of an Antoine construction as in Figure \ref{Antoine}. The second lemma shows that homeomorphisms of $S^{3}$ taking $\mathcal{C}$ to itself necessarily fix the point $p$. The third lemma shows that homeomorphisms of $S^{3}$ taking some $C(v(g))$ to some $C(v(g^{\,\prime}))$, when restricted to $C(v(g))$, agree with some
$H_{g_{1}}$ restricted to $C(v(g))$.

\begin{lem}
\label{TechLemma}
Let $C_1$, $C_2$ and $C_3$  be  rigid Antoine Cantor sets in $\mathbb{R}^3$. 
Let $T_i$ and $T^{\prime}_i$ be the first and second stages in the construction of $C_i$. 
Assume that $T_1$ and $T_2$ are linked (and thus are  disjoint). 
Let $ h:\mathbb{R}^3\rightarrow \mathbb{R}^3$ be a homeomorphism such that 
\begin{itemize}
\item[\textbf{**}]$h(C_1) \cap C_3 \neq \emptyset$.
\end{itemize}
Assume in addition that
\begin{itemize}
\item[\textbf{a)}] 
$h\left (\frac{\mathstrut}{\mathstrut} \partial(T_1\cup T_2) \cup \partial(T^{\prime}_1\cup T^{\prime}_2)\right )$  
$\cap$ 
 $\left (\frac{\mathstrut}{\mathstrut}\partial(T_3)\cup \partial(T^{\prime}_3)\right ) = \emptyset$, 
\item [\textbf{b)}]
$h(C_1)\cap T_3\subset C_3$, and %
\item[\textbf{c)}] 
$h(C_2)\cap T_3\subset C_3$.
\end{itemize}
If $h(T_1)\subset T_3$, then 
\begin{itemize}
\item[\textbf{I)}]either 
the geometric index of  $h(T_1)$ in $T_3$ is 0,  $h(T_2)$ is also contained in $T_3$, and the geometric index of $h(T_2)$ in $T_3$ is also 0, 
\item[\textbf{II)}] or the geometric index of $h(T_1)$ in $T_3$ is $1$ and $C_3\subset h(T_1)$. 
\end{itemize}
\end{lem}

\begin{proof}
We have $h(T_1)\subset T_3$ and we consider the geometric index of $h(T_1)$ in $T_3$.

\textbf{Index 0: }Assume that the geometric index is $0$. Then $h(T_1)$ is contained in a cell $B$ in $T_3$, and so contracts in $T_3$.

If $\partial(h(T_2))$ does not intersect $\partial(B)$, then either $h(T_2)\subset B$, $B\subset h(T_2)$ or $h(T_2)$ misses $B$. The last case cannot occur because of the linking of $h(T_2)$ and $h(T_1)$. In the first case, $h(T_2)\subset T_3$ and the geometric index of  $h(T_2)$ in $T_3$ is $0$. So condition \textbf{(I)} holds. The second case cannot occur because $h(T_1)$ and $h(T_2)$ are disjoint.

If $\partial(h(T_2))$ does  intersect $\partial(B)$, then $h(T_2)\subset T_3$ by \textbf{(a)} above. If the geometric index of $h(T_2)$ in $T_3$ is $\geq 1$, then $h(T_2)$ cannot be contained in any component of $T^{\prime}_3$ by Theorem \ref{productindex}. Thus, by \textbf{(a)}, $h(T_2)$ contains any component of $T^{\prime}_3$ that it intersects. By \textbf{(c)}, $h(C_2)\cap C_{3}\neq \emptyset$. So $h(T_2)$ intersects and thus contains some component $D_1$ of $T^{\prime}_3$. The geometric index of $D_1$ in $h(T_2)$ must be $0$ by Theorem \ref{productindex}. Thus $D_1$ is contained in a cell in $h(T_2)$ and contracts in $h(T_2)$.

Let $D_1, D_2,\ldots D_m$ be the components of $T^{\prime}_3$, listed so that $D_i$ links $D_{i+1}$. Since $D_1$ and $D_2$ are linked, the Antoine construction guarantees that any contraction of $D_1$ intersects $C_3\cap D_2$. So $h(T_2)$ intersects and thus contains $D_2$ and $D_2$ has geometric index 0 in $h(T_2)$. Continuing inductively, all of the $D_i$ are contained in $h(T(2))$ and thus $h(T_2)$ contains $C_3$.

Since $h(T_1)$ is disjoint from $h(T_2)$, this contradicts condition \textbf{**}.

\textbf{Index greater than 1:}
If the geometric index of $h(T_1)$ in $T_3$ is greater than $ 1$, then by Theorem \ref{productindex}, $h(T_1)$ cannot be contained in any component of $T^{\prime}_3$.  So $h(T_1)$ contains each component of $T^{\prime}_3$ that it intersects. Let $D_1, D_2,\ldots D_m$ be the components of $T^{\prime}_3$, listed so that $D_i$ links $D_{i+1}$. By condition \textbf{**}, $h(T_1)$ intersects and thus contains some component of $T^{\,\prime}_{3}$, say $D_1$.
The geometric index of $D_1$ in  $h(T_1)$ is 0 by Theorem \ref{productindex}. 

By an argument similar to that in the \textbf{Index 0} case above, $h(T_1)$ contains each $D_i$, and each $D_i$ is of index 0 in $h(T_1)$.
By Theorem \ref{evenindex}, the geometric index of $T^{\prime}_3$ in  $h(T_1)$ is even. Since the geometric index of  $T^{\prime}_3$ in  $T_3$ is two \cite{AGRW17}, the geometric index of $T^{\prime}_3$ in  $h(T_1)$ cannot be $0$ and so is at least two.
Theorem \ref{productindex} now implies that the geometric index of  $T^{\prime}_3$ in  $T_3$ is at least 4, which is a contradiction. So the the geometric index of $h(T_1)$ in $T_3$ cannot be greater than $ 1$.

\textbf{Index equal to 1:}
If the geometric index of $h(T_1)$ in $T_3$ is $1$, then by Theorem \ref{productindex}, $h(T_1)$ cannot be contained in any component of $T^{\prime}_3$.  An argument similar to that in the previous two cases shows that $h(T_1)$ contains each component of $T^{\prime}_3$.
Thus $C_3\subset h(T_1)$. So condition \textbf{(II)} holds. 
\end{proof}
\begin{lem}
\label{GenusLemma}
Let $\Gamma:\mathcal{C}\rightarrow \mathcal{C}$ be a homeomorphism that extends to a homeomorphism of $S^3$. Then $\Gamma(p)=p$.
\end{lem}
\begin{proof}
The local genus of any point in $\mathcal{C}$ other than $p$ is less than or equal to one because these points are defined by nested sequences of tori which are of genus one. If we can show the local genus of $p$ is at least 2, then the results mentioned in Remark \ref{genus}  show that 
$\Gamma (p)=p$. 
\begin{remark}\label{localgenus}
The local genus of $p$ in $\mathcal{C}$ cannot be $0$ because otherwise, there would be arbitrarily small 3-balls containing $p$ with boundary missing 
$\mathcal{C}$. These boundary 2-spheres  in the complement of 
$\mathcal{C}$ would separate linked stages of some Antoine construction. This cannot happen.
\end{remark}

\begin{figure}[h]
\begin{center}
    \includegraphics[width=0.47\textwidth]{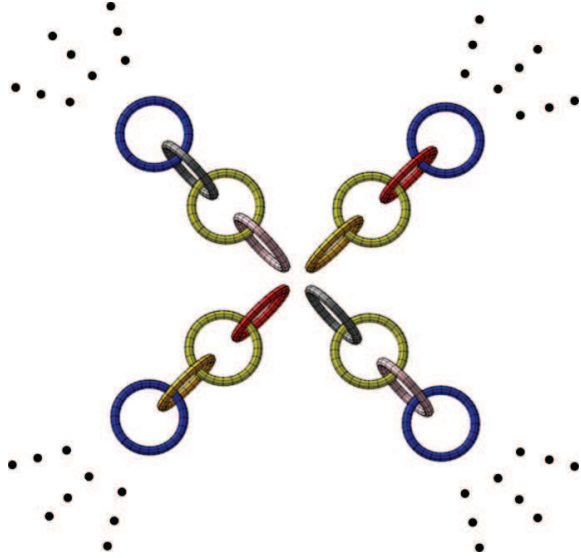}
\end{center}
 \caption{%
    Start of construction for $\mathcal{C}_{1}$ and $\mathcal{C}_{2}$ in the case F(2)}%
\label{rigid2}
\end{figure}

To see that the local genus of $p$ in $\mathcal{C}$ is at least 2, consider 2 subsets of $\mathcal{C}$ that meet in $p$. These subsets are obtained by   dividing
$\mathcal{C}-C(v(e))$ into two subsets that meet in $p$.

 For a nonidentity $g\in F_{N}$, with  unique reduced representation  as $\beta_{i1}\circ \beta_{i2}\circ \ldots \circ \beta_{iM}$ where the $\beta_{ij}$ are in $g(F_{N})$, let $g_1=\beta_{i1}$.

$$ \text{Let }\mathcal{C}_{1}=
 \bigcup_{k=1}^N \bigg ( \{C(v(g)),  g_{1}=\beta_{k}  \}
\cup \bigcup_{i=1}^{3}
\{C_{i}(v(g),\beta_j), g_{1}=\beta_{k} \}\bigg )
\cup\{p\}, 
\text{\ and  } $$

$$ \mathcal{C}_{2}=
\displaystyle \bigcup_{k=N+1}^{2N} \bigg ( \{C(v(g)),  g_{1}=\beta_{k}  \}
\cup \bigcup_{i=1}^{3}
\{C_{i}(v(g),\beta_j), g_{1}=\beta_{k} \}\bigg )
\cup\{p\}.$$

Figure \ref{rigid2} illustrates the start of the described construction in the case $N=2$. The Cantor set $\mathcal{C}_{1}$ arises from tori in the upper half of the figure, and the Cantor set set $\mathcal{C}_{2}$ arises from tori in the lower half of the figure.

Note that $\mathcal{C}_{1}$ and $\mathcal{C}_{2}$ intersect in $p$. The local genus of $p$ in each of these $2$ Cantor sets is at least 1 by reasoning similar to that in Remark \ref{localgenus}.

The $\mathcal{C}_{i}$  satisfy the conditions of Theorem \ref{Slicing}. It follows that the genus of $p$ in $\mathcal{C}$ is $\geq$ 2. 

\emph{Note:} In fact, the local genus of $p$ in $\mathcal{C}$ is $\infty$.
\end{proof}

\begin{lem}
\label{homeomorphism-determination}
If $\Gamma$ is a self homeomorphism of $S^3$ and $\Gamma (C(v(g)))=C(v(g^{\,\prime}))$ for some $g, g^{\,\prime}\in F_{N}$ with $g^{\,\prime}=g\circ g_{1}$ then $\Gamma\vert_{C(v(g))}=H_{g_{1}}\mathstrut\vert_{C(v(g))}$. 
In particular, if $\Gamma (C(v(e)))=C(v(g))$ for some $g\in F_{N}$, then $\Gamma\vert_{C(v(e))}=H_{g}\vert_{C(v(e))}$. 
\end{lem}

\begin{proof}

$H_{g_1}(C(v(g)))=H_{g_1}\left( H_{g}(C(v(e)))\right)=H_{g_{1}\circ g}(C(v(e)))=
H_{g^{\,\prime}}(C(v(e)))=C(v(g^{\,\prime}))$. 
So both $H_{g_{1}}$ and $\Gamma$ take $C(v(e))$ onto $C(v(g^{\,\prime}))$. 
Thus $H_{g_{1}}\circ \Gamma^{-1}$ takes $C(v(e))$ onto $C(v(e))$. The rigidity of $C(v(e))$ implies that $H_{g_1}\vert_{C(v(e))}=\Gamma\vert_{C(v(e))}$.
\end{proof}

\begin{thm}[Main Theorem]
$\mathcal{M}$ is an irreducible 3-manifold with end homogeneity group $F_{N}$.
\end{thm}

This follows immediately from the next theorem.

\begin{thm}
\label{MT 2}
$\mathcal{C}\subset S^3$ is unsplittable and the embedding homogeneity group of $\mathcal{C}\subset S^3$ is $F_{N}$.
\end{thm}
\begin{proof}
The unsplittability was addressed above. Let $\mathcal{H}$ be the embedding homogeneity group of 
$\mathcal{C}$. Note that the homeomorphisms $H_{g}:\mathbb{R}^3\rightarrow \mathbb{R}^3, g\in F_{N}$, extend to homeomorphisms
$H^{\prime}_{g}:S^3\rightarrow S^3, g\in F_{N}$, by fixing the point at infinity.

For each $g\in F_{N}$, let $\eta_{g}=H^{\prime}_{g}\vert_{\mathcal{C}}$. Then each $\eta_{g}$ takes $\mathcal{C}$ onto $\mathcal{C}$. This follows directly from the definition of $C(v(g))=H_{g}(C(v(e)))$ and 
$C_{i}(v(g),\beta_k)=H_{g}(C_{i}(v(e),\beta_{k}))$. The elements $\{\eta_{g}, g\in F_{N}\}$ form a subgroup of $\mathcal{H}$ since
$H^{\prime}_{g_1}\circ H^{\prime}_{g_2}=H^{\prime}_{g_1\circ g_2}$ by Remark \ref{associativity}.

Let $\gamma$ be any self homeomorphism of $\mathcal{C}$ that extends to a homeomorphism $\Gamma$ of $S^3$. We will show that  $\gamma =\eta(g)$ for some $g\in F_{N}$.

\textbf{Step 1}: We first show that $\Gamma (C(v(e))) = C(v(g))$ for some $g\in F_{N}$. If this is true, then  the definition of $C(v(g))$ as $H_{g}(C(v(e)))$ and Lemma \ref{homeomorphism-determination} shows that $\Gamma \vert_{C(v(e))}=\eta_{g}\vert_{C(v(e))}$ Let:
\begin{align*}
\Lambda_{N}=
\{ T(C(v(g)))\ \vert \parallel v(g)\parallel \leq N\}&\ \bigcup\ 
\{ T(C_{j}(v(g), \beta_{k}))\ \vert \parallel v(g)\parallel \leq N\},\\
\Delta_{N}=
\{C(v(g)),  \vert \parallel v(g)\parallel \leq N \}&\ \bigcup\ 
\{ C_{j}(v(g), \beta_{k})\ \vert \parallel v(g)\parallel \leq N\}, \text{and}\\
\Lambda^{\prime}_{N}=
\{ T^{\prime}(C(v(g)))\ \vert \parallel v(g)\parallel \leq N\}&\ \bigcup\ 
\{ T^{\prime}(C_{j}(v(g), \beta_{k}))\ \vert \parallel v(g)\parallel \leq N\}.
\end{align*}

$\Gamma(C(v(e)))$ does not contain
$p$ by Lemma \ref{GenusLemma}. So there is a positive  integer $N_{1}$ such that
$\Gamma(C(v(e)))\subset \Delta_{N_{1}}$.
 Similarly, there is a positive integer $N_{2}>N_{1}$ such that
$\Gamma^{-1}(\Delta_{N_{1}})\subset \Delta_{N_{2}}$.

Using techniques similar to those in \cite{Sh68}, or \cite{GRW14},
  choose a homeomorphism $k$ of $S^{3}$ to itself, fixed on $\mathcal{C}$,
 so that 
 \begin{align}
 k
 \left(
\Gamma
 \left(
 \partial (\Lambda_{N_2+1})
 \cup
\partial(\Lambda^{\prime}_{N_2+1})
\right)
\right)
\cap
\left(
\partial (\Lambda_{N_2+1})
\cup
\partial( \Lambda^{\prime}_{N_2+1})
\right)
 =\emptyset.
\label{boundary}
 \end{align}
Let $\Gamma^{\,\prime}=k\circ \Gamma$. Let $c$ be a point of $C(v(e))$ and let $\Gamma^{\,\prime}(c)=\Gamma(c)=q\in \widetilde{C}$ where $\widetilde{C}=C(v(g))$ (or $= C_{i}(v(g),\beta_k)$) for some $g$ and where $T=T(C(v(g)))$ (or $= T(C_{i}(v(g),\beta_k))$).
We will show that in fact, $\Gamma^{\,\prime}(C(v(e)))=\Gamma(C(v(e)))=C(v(g))$ for some $g$, and that the second case above cannot occur.  Observe that
either $\Gamma^{\,\prime}(T(C(v(e))))\subset {\rm{Int}} T$ or
${\rm{Int}}(\Gamma^{\,\prime}(T(C(v(e)))))\supset T$ by Equation \ref{boundary} above.

\textbf{Step 1a:} 
Assume that $\Gamma^{\,\prime}(T(C(v(e))))\subset {\rm{Int}}(T)$. Lemma \ref{TechLemma} now applies with $T(C(v(e)))=T_1, \text{ any }T_{1}(v(e),\beta_i)=T_2, \text{ and with } T=T_3$. If $\Gamma^{\,\prime}(T(C(v(e))))$ has
geometric index $0$ in $T_{3}$, choose a linked chain of tori in the construction of $C(G(F_{N}))$ joining $T(C(v(e)))$ to a $T(C(v(g^\prime)))$ where $|| g^\prime || = N_2+1$. By repeated applications of Lemma \ref{TechLemma} it follows
that $\Gamma^{\,\prime}(T(C(v(g^\prime))))\subset \Delta_{N_1}$ which is a contradiction.
So   $\Gamma^{\,\prime}(T(C(v(g))))$ has geometric index greater than or equal to $ 1$ in $T$. By  Lemma \ref{TechLemma}, this geometric index is 1 and  $\widetilde{C}\subset  \Gamma^{\,\prime}(T(C(v(g))))$.

Since $\Gamma(\mathcal{C})=\mathcal{C}$,
we have  $\widetilde{C}=\Gamma(C(v(e))$.
Since $C(v(e))$ in not equivalent to any $C_{i}(v(g),\beta_k)$  by Remarks \ref{inequivalent} and \ref{inequivalent2}
we get 
$\widetilde{C}=C(v(g))$ for some $g$. It follows from Lemma \ref{homeomorphism-determination} that $\Gamma \vert_{C(v(e))}=\eta_{g}\vert_{C(v(e))}$.

\textbf{Step 1b:}  Assume that ${\rm{Int}}(\Gamma^{\,\prime}(T(C(v(e)))))\supset {\rm{Int}}T$. 
Then ${(\Gamma^{\,\prime})^{-1}}(T)\subset{\rm{Int}}(T(C(v(e)))$. The argument
from Case I can now be repeated replacing $\Gamma^{\,\prime}$ by $(\Gamma^{\,\prime})^{-1}$
and interchanging $T$ and $T(C(v(e)))$. It follows that
$(\Gamma^{\,\prime})^{-1}(\widetilde{C})=C(v(e))$ and so $\Gamma^{\,\prime}(C(v(e))=C(v(g))$ for some $g$ as claimed.
Again it follows from Lemma \ref{homeomorphism-determination} that $\Gamma \vert_{C(v(e))}=\eta_{g}\vert_{C(v(e))}$.

\textbf{Step 2}:  By an argument similar to that in Step 1, for each $g_i\in F_{N},$ $\Gamma (C(v(g_i)))$ $=C(g_{m(i)}\circ g_i)$ for some $g_{m(i)}\in F_{N}$.  Also for each $i$, $j$ and $k$,  
\\
\centerline{$\Gamma (C_{j}(v(g_i)\beta_k))$ $=C_{j}(v(g_{m(i,k)}\circ g_i),\beta_k)$ for some $g_{m(i,k)}\in F_{N}$.}
Working inductively outward from $C(v(e))$, and using the fact that if $T_1$ and $T_2$ are linked tori, then $\Gamma(T_1)$ and $\Gamma(T_2)$ are linked, one sees that each $g_{m(i)}=g$ and each $g_{m(i,k)}=g$.

By Lemma \ref{homeomorphism-determination} it follows that each $\Gamma
\vert_{C_{j}(v(g_i)\beta_k)}=\eta_{g}\vert_{C_{j}(v(g_i)\beta_k)}$, and that
each  $\Gamma
\vert_{C(v(g_i))}=\eta_{g}\vert_{C(v(g_i))}$
It now follows that $\Gamma =\eta(g)$ for the $g$ from Steps 1 and  2.
This shows that $\{\eta(g), g\in F_{N}\}$ is in fact, the group $\mathcal{H}$ and that $\mathcal{H}$ is isomorphic to $F_{N}$. 
\end{proof}

The method used in the construction of $\mathcal{C}$ can be used for any Cayley graph of a finitely generated group $F$ for which the graph automorphisms have certain extension properties. This observation is given in the next theorem.

\begin{thm}
\label{generalization}Suppose $F$ is a finitely presented group with a Cayley Graph $G(F)$ in $S^3$. Suppose the graph automorphisms $h_g$ for $g\in F$ can be extended to homeomorphisms $H_g$ of $S^3$. Suppose also
that for every pair of group elements $g_1, g_2$,  $H_{g_1}\circ H_{g_2} = H_{g_1\circ g_2}$. Then there is an open irreducible 3-manifold $M_F$ with embedding homogeneity group isomorphic to $F$.
\end{thm}
\begin{proof}
The construction of a collection of Cantor sets modeled on the graph, and the proof  analogous to the proof of  Theorem \ref{MT 2} go through in this case.
\end{proof}

\section*{Questions}

In \cite{GaRe14}, the authors showed that for every finitely generated abelian group $G$, there is an irreducible open 3-manifold $M_G$ with end homogeneity group isomorphic to $G$.  In this paper, we show that for any finitely generated free group $F_{N}$, there is an irreducible open 3-manifold $M_{F_{N}}$ with end homogeneity group $F_{N}$. 

\begin{que2}[\textbf{1}]
Which non-abelian finitely generated groups can arise as end homogeneity groups of (irreducible) open 3-manifolds?
\end{que2}

In light of Theorem \ref{generalization}, this leads to the following question.

\begin{que2}[\textbf{2}]
Which non-abelian finitely generated groups have Cayley graphs embeddable in $S^3$ with the properties listed in Theorem \ref{generalization}?
\end{que2}

A related question is  the following.

\begin{que2}[\textbf{3}]
If not all finitely generated non-abelian groups have Cayley graphs with the properties listed in Theorem \ref{generalization}, is there a way to characterize which ones do have this property?
\end{que2}

\section*{Appendix}
This section lists some of the important notation used in the paper.

\begin{itemize}
\item
$F_{N}$ is the free group on N generators.
\item
$g(F_{N}) = \{a_1,\ldots,a_N, \overline{\mathstrut a_1},\ldots, \overline{\mathstrut a_N}\}$, is the full set of generators for $F_{N}$. 
\item $\beta_j = \alpha_j$ for $1\leq j \leq N$ and  $\beta_j =\overline{\mathstrut \alpha_{j-N}}$ for $N+1\leq j \leq 2N$.
\item $G(F_{N})$ is the Cayley graph constructed in $\mathbb{R}^3$ for  $F_{N}$.
\item $v(e), v(\beta_j), v(g)$ are vertices of $G(F_{N})$ associated with group elements.
\item $E(v(g), \beta_k) =E((v(g),v(g\circ \beta_k)))$ is the oriented edge labeled $\beta_k$ from $v(g)$ to $v(g\circ \beta_k)$.\\ 
This edge is the regarded the same as  
$E(v(g\circ \beta_k), \overline{\beta_k}) =E(v(g\circ \beta_k),v(g))$.
\item $h_{g}$ is the automorphism  of $G(F_{N})$ given by $h_{g}(v(g_i))=v(g\circ g_i)$,\\
$h_{g}(E(v(g_i), \beta_k))=E(v(g\circ g_i),\beta_k)$.
\item $H_g$ is the extension of $h_{g}$ to a homeomorphism of $\mathbb{R}^3$.
\item $B(v(g))=H_{g}(B(v(e)))$ is a 3-ball about the vertex $v(g)$.
\item $D(g,\beta_k)=H_{g}(D(e,\beta_k))$ is a tube joining $B(v(g))$ to $B(v(g\circ \beta-k))$.
\item $C(v(g))=H_{g}(C(v(e))$ is a rigid Cantor set in $B(v(g))$.
\item $C_{1}(v(g),\beta_k), C_{2}(v(g),\beta_k), and C_{3}(v(g),\beta_k)$, form a linked chain of rigid Cantor sets joining $C(v(g))$ to $C(v(g\circ \beta_k))$. $C_{i}(v(g),\beta_k)=H_{g}(C_{i}(v(e)), \beta_k)$.
\item $C(G(F_{N}))=\mathcal{C}$. $M_{F_{N}}=\mathcal{M}$.
\end{itemize}

\section*{Acknowledgements}
The authors would like to thank the referee for many helpful suggestions and comments. These suggestions helped to clarify the arguments in the paper.
Both authors were
supported in part by the Slovenian Research Agency grant
BI-US/19-21-024. The second author was supported in part by the
Slovenian Research Agency grants  P1-0292, N1-0114, N1-0083, N1-0064, and J1-8131.


\def\cprime{$'$}






\end{document}